\documentclass[11pt,reqno]{amsart}

\usepackage{graphics,graphicx,fontenc,mathabx,wasysym,cancel}
\usepackage{epsfig,psfrag,manfnt}
\usepackage{bbm,subfigure,rotating,bm,float,mathdots,mathrsfs,slashed}
\usepackage[all,cmtip]{xy}
\usepackage{amssymb,amsmath,amsfonts,amsthm,color}

\usepackage{scalerel,stackengine}
\stackMath
\newcommand\reallywidehat[1]{%
\savestack{\tmpbox}{\stretchto{%
  \scaleto{%
    \scalerel*[\widthof{\ensuremath{#1}}]{\kern-.6pt\bigwedge\kern-.6pt}%
    {\rule[-\textheight/2]{1ex}{\textheight}}
  }{\textheight}%
}{0.5ex}}%
\stackon[1pt]{#1}{\tmpbox}%
}

\newcommand{\calD}{\mathcal{D}}
\newcommand{\calE}{\mathcal{E}}
\newcommand{\calF}{\mathcal{F}}

\newcommand{\calK}{\mathcal{K}}
\newcommand{\calL}{\mathcal{L}}

\newcommand{\calS}{\mathcal{S}}

\newcommand{\mA}{\mathbb{A}}

\newcommand{\mC}{\mathbb{C}}

\newcommand{\mN}{\mathbb{N}}

\newcommand{\mR}{\mathbb{R}}

\newcommand{\mZ}{\mathbb{Z}}

\newcommand{\bbe}{\mathbf{e}}

\newcommand{\bbk}{\mathbf{k}}

\newcommand{\bbn}{\mathbf{n}}

\newcommand{\bbx}{\mathbf{x}}
\newcommand{\bby}{\mathbf{y}}
\newcommand{\bbz}{\mathbf{z}}

\newcommand{\bbX}{\mathbf{X}}

\newcommand{\beeta}{\bm{\eta}}

\newcommand{\bsigma}{\bm{\sigma}}

\newcommand{\btau}{\bm{\tau}}
\newcommand{\bxi}{\bm{\xi}}

\newcommand{\mybinom}[2]{\Bigl(\begin{array}{@{}c@{}}#1\\#2\end{array}\Bigr)}

\newtheorem{theorem}{Theorem}[section]
\newtheorem{lemma}[theorem]{Lemma}
\newtheorem{corollary}[theorem]{Corollary}
\newtheorem{proposition}[theorem]{Proposition}

\theoremstyle{definition}

\newtheorem{example}[theorem]{Example}

\theoremstyle{definition}
\newtheorem{definition}[theorem]{Definition}

\theoremstyle{definition}

\begin{document}

\keywords{partial differential equations, null solutions,
  distributions that are tempered in the spatial directions}

\subjclass[2010]{Primary 35A02; Secondary  58J40, 35A24, 35A22, 46F12, 32C25}

\title[Existence of null solutions]{On the existence of 
\\spatially tempered null solutions to \\
linear constant coefficient PDEs}

\author[]{Amol Sasane}
\address{Mathematics Department, London School of Economics, 
Houghton Street, London WC2A 2AE, U.K.}
\email{A.J.Sasane@lse.ac.uk}

\begin{abstract} 
Given a linear, constant coefficient partial differential equation 
in $\mR^{d+1}$, where one independent variable plays the role of
`time', a distributional solution is called a null solution if its
past is zero. Motivated by physical considerations, we consider
distributional solutions that are tempered in the spatial
directions alone (and do not impose any restriction in the time direction). 
Considering such spatially tempered distributional
solutions, we give an algebraic-geometric characterization, in terms
of the polynomial describing the PDE at hand, for the null solution
space to be trivial (that is, consisting only of the zero
distribution).
\end{abstract}

\maketitle

\section{Introduction}

\noindent 
Given a polynomial $p\in \mC[X_1,\cdots, X_d,T]=:\mC[\bbX,T]$, we
associate with it the linear constant coefficient differential
operator $ D_p$ by making the replacements
\begin{eqnarray*}
X_k&\rightsquigarrow&\frac{\partial}{\partial x_k},\quad k=1,\cdots,d\\
T&\rightsquigarrow& \frac{\partial}{\partial t}.
\end{eqnarray*}

\begin{definition}[Solution space]$\;$

\noindent 
A {\em solution space} is a subspace $S$ of the space of distributions
$\calD'(\mR^{d+1})$.
\end{definition}

\noindent Throughout this article, unless otherwise indicated, we will
use the standard distribution theory notation from Schwartz \cite{Sch}
or Tr\'eves \cite{Tre}.  If we fix a solution space $S$, a polynomial
$p\in \mC[\bbX,T]$ gives rise to the differential operator
$D_p :S\rightarrow \calD'(\mR^{d+1})$, defined by
$$
D_p u:= p\left( \frac{\partial}{\partial x_1},\cdots,
  \frac{\partial}{\partial x_d},\frac{\partial}{\partial t}\right)
u,\quad u\in S.
$$

\goodbreak

\begin{definition}[Null solution, null solution space]$\;$

\noindent 
Let $p\in \mC[\bbX,T]$, and $S$ be a solution space.  A {\em null
  solution in $S$ associated with $p$} is a distribution $u\in S$ such
that
\begin{itemize}
 \item (Solution) $\;D_p u=\mathbf{0}$
 \item (Null past) $u|_{t<0}=\mathbf{0}$
\end{itemize}
We denote by $N_S(p)$ the subspace of $S$ consisting of all null solutions
in $S$ associated with $p$:
$$
N_S(p):=\{u\in S:  D_pu=\mathbf{0}\textrm{ and  }u|_{t<0}=\mathbf{0} \}.
$$
\end{definition}

\noindent The notion of a null solution was considered in
\cite{Hor0} and \cite{Hor3}.  

We are interested in giving an
algebraic-geometric characterisation of the polynomials $p$ for which
$N_S(p)$ is just the trivial subspace $\{\mathbf{0}\}$, consisting of
only the zero distribution $\mathbf{0}$. Such a characterization is 
 expected to depend on the solution space $S$, as illustrated by the two 
results below, Propositions~\ref{prop_a} and \ref{prop_Hor}. In the following, $\calD(\mR^{d+1})$ denotes the space of all 
 compactly supported smooth ($\textrm{C}^\infty$) functions on $\mR^{d+1}$, and 
 $\calE'(\mR^{d+1})$ denotes the space of all compactly supported distributions on $\mR^{d+1}$. 

\begin{proposition}
\label{prop_a}
 Let $S=\calE'(\mR^{d+1})$ or $S= \calD(\mR^{d+1})$.  Let $p\in \mC[\bbX,T]$. 

 \noindent Then $N_{S}(p)=\{\mathbf{0}\}$ if and only if  $p\neq \mathbf{0}$.
\end{proposition}

\goodbreak 

\begin{proof} (`If' part): Suppose that $D_pu=\mathbf{0}$ and that $p\neq \mathbf{0}$. By the Payley-Wiener-Schwartz theorem 
\cite[Prop. 29.1,p. 307]{Tre}, the Fourier transform $\calF u$ of $u$ (with respect to {\em all} the variables) can be extended 
to an entire function on $\mC^{d+1}$. Thus $D_pu=\mathbf{0}$ yields that for all $\bbz\in \mC^{d+1}$, 
$$
p(i\mathbf{z}) \cdot (\calF u)(\mathbf{z})=0.
$$
But the ring $\textrm{A}(\mC^{d+1})$ of entire functions in $d+1$ complex variables forms an integral domain.  
As $p(i\cdot)\neq \mathbf{0}$ in $\textrm{A}(\mC^{d+1})$, we conclude that $\calF u=\mathbf{0}$, and so $u=\mathbf{0}$ too. 
Hence $N_S(p)=\{\mathbf{0}\}$. 

\medskip 

\noindent (`Only if' part): Suppose that $p=\mathbf{0}$. Then clearly $N_S(p)=S\neq \{\mathbf{0}\}$. 
\end{proof}

\noindent 
On the other hand,  when $S=\textrm{C}^\infty(\mR^{d+1})$ or
$\calD'(\mR^{d+1})$, using two results 
due to H\"ormander \cite[Theorems~8.6.7, 8.6.8]{Hor}, one can show the result below. 
Here, $\deg(\cdot)$ is used to denote the {\em total degree},
which is the maximum (over the monomials
$X_1^{i_1} \cdots X_d^{i_d} T^{i_{d+1}}$ occurring in the polynomial)
of the sum of the degrees of the exponents of indeterminates in the
monomial. Also if we decompose $p=p_m+p_{m-1}+\cdots+p_0$, where $p_j$ is homogeneous of total degree $j$ and $p_m\neq \mathbf{0}$, 
then $p_m$ is called the {\em principal part} of $p$. Let $\langle\cdot,\cdot\rangle_{\mR^{d+1}}$ denote the Euclidean inner product on $\mR^{d+1}$. 
Given a vector $\bbn\in \mR^{d+1}$, the plane $\{\bby:\langle \bby,\bbn\rangle_{\mR^{d+1}}=0\}$   
is called {\em characteristic with respect to $D_p$} 
if $p_m(\bbn)=0$. 

\begin{proposition}
\label{prop_Hor}
Let $S=\calD'(\mR^{d+1})$ or $\textrm{\em C}^\infty(\mR^{d+1})$ and let $p\in
\mC[\bbX,T]$.  Then $N_{S}(p)=\{\mathbf{0}\}$ if and only if  $
\deg(p(\bbX,T))=\deg(p(\mathbf{0},T))$.
\end{proposition}
\begin{proof} 
We recall that \cite[Theorem~8.6.7]{Hor} says that for a characteristic plane with normal $\bbn$, there exists a solution 
 in $\textrm{C}^\infty$ whose support is exactly the half-space $\{\bby:\langle \bby, \bbn\rangle_{\mR^{d+1}}\leq 0\}$. 
 It can be seen that the hyperplane with the normal vector $\bbn:=(\mathbf{0},1)\in \mR^{d+1}$ 
is characteristic with respect to $D_p$ if and only if $\deg(p)\neq \deg(p(\mathbf{0},T))$. 
 This immediately gives the `only if' part of the proposition. 
 
 For the `if' part, we use \cite[Theorem 8.6.8]{Hor}, which says that if $X_1,X_2$ are open convex sets such that $X_1\subset X_2$, 
 then the following are equivalent:
 
 \smallskip 
 
 \noindent $\;\bullet$  If $u\in \calD'(X_2)$ satisfies $D_p u=\mathbf{0}$ in $X_2$ and $u|_{X_1}=\mathbf{0}$, then $u=\mathbf{0}$ in $X_2$. 
  
  \noindent $\;\bullet$ Every characteristic hyperplane which intersects $X_2$ also intersects $X_1$. 
 
 \smallskip 
 
 \noindent 
Taking $X_1=\{(\bbx,t):\langle (\bbx, t),\bbn\rangle_{\mR^{d+1}}=t<0\}$, where $\bbn:=(\mathbf{0},1)\in \mR^{d+1}$, and with $X_2:=\mR^{d+1}$,  
 the above  yields the `if' part of the proposition.  
\end{proof}

\noindent 
Let us see what happens when we apply Proposition~\ref{prop_Hor} in
the case of the diffusion equation.

\begin{example}[Diffusion equation] Consider the diffusion equation
$$
\Big(\frac{\partial}{\partial t}-\Delta \Big)u=\mathbf{0},
$$
that is, $D_pu=\mathbf{0}$, where $ p(\bbX,T)=T -
(X_1^2+\cdots+ X_d^2)$.  

\smallskip 

\noindent Since $\deg(p(\bbX,T))=2$, whereas
$\deg(p(\mathbf{0},T))=\deg(T)= 1$, Proposition~\ref{prop_Hor} implies
that $N_{\calD'(\mR^{d+1})}(p)\neq \{\mathbf{0}\}$ and
$N_{\textrm{C}^\infty(\mR^{d+1})}(p)\neq \{\mathbf{0}\}$.
\hfill$\Diamond$
\end{example}

\noindent 
In the above example, the outcome is physically unexpected. Indeed,
matter diffusion can be modelled by the PDE above, where $u$
is the density of matter. Then zero density up to time $t=0$
should mean that the density stays zero in the future as well.
However, the above example shows that one can have `pathological' null
solutions in $\textrm{C}^\infty$ or in $\calD'$ that are nonzero in the future. 
On the other hand, if we choose a
different, physically motivated solution space in this context, namely
functions which at each time instant have a spatial profile belonging
to $\textrm{L}^1(\mR^d)$, then the associated null solution space is
trivial, as expected.  The reason that the null solution space is
nontrivial in the above example when $S=\textrm{C}^\infty(\mR^{d+1})$
or $\calD'(\mR^{d+1})$ is that there is no growth restriction on the
spatial profiles of the solutions at each time instant, and `rapid'
growth\footnote{Roughly speaking, faster than $e^{\|\bbx\|^2}$; see
  \cite[Theorem, p.44]{Hel}} is allowed.  Indeed, in most physical
situations, we expect that at each time instant, the spatial profile
is typically in some $\textrm{L}^p$ space or at most
polynomially growing, etc.  This motivates the following solution
space considered in this article. Below $\calS(\mR^d)$ denotes the 
Schwartz space of test functions, and $\calS'(\mR^d)$ denotes the space of tempered 
distributions; see e.g. \cite[Chap. 25]{Tre}.

\goodbreak

\begin{definition}[Distributions tempered in the spatial directions]$\;$

\noindent 
The {\em space of distributions on $\mR^{d+1}$ tempered in the spatial
  directions}, is the space $\calL(\calD(\mR), \calS'(\mR^d))$ of all
continuous linear maps from $\calD(\mR)$ to $\calS'(\mR^d)$, where
$\calD(\mR)$ is endowed with its inductive limit topology and
$\calS'(\mR^d)$ is equipped with the weak dual topology
$\sigma(\calS',\calS)$.  We endow $\calL(\calD(\mR), \calS'(\mR^d))$
with the topology $\calL_\sigma(\calD(\mR),\calS'(\mR^d))$ of
pointwise convergence. 

\smallskip 

\noindent For $u\in \calL(\calD(\mR), \calS'(\mR^d))$, 
 $k=1,\cdots, d$, we define 
 $\displaystyle 
 \frac{\partial u}{\partial x_k} \in \calL(\calD(\mR), \calS'(\mR^d))
 $ 
 by 
 $$
 \left\langle \frac{\partial u}{\partial x_k} (\varphi),\psi\right\rangle:=-\left\langle u(\varphi),  \frac{\partial \psi}{\partial x_k} \right\rangle 
 \quad (\varphi \in \calD(\mR),\;\psi\in \calS(\mR^d)).
 $$
 For $u\in \calL(\calD(\mR), \calS'(\mR^d))$, we define 
  $\displaystyle 
 \frac{\partial u}{\partial t} \in \calL(\calD(\mR), \calS'(\mR^d))
 $ 
 by 
 $$
 \left\langle \frac{\partial u}{\partial t} (\varphi),\psi\right\rangle:=-\left\langle u(\varphi'),  \psi \right\rangle 
 \quad (\varphi \in \calD(\mR),\;\psi\in \calS(\mR^d)).
 $$
 Also, for an element $u\in \calL(\calD(\mR), \calS'(\mR^d))$, we
define its `spatial' Fourier transform by
$\widehat{u}\in \calL(\calD(\mR), \calS'(\mR^d))$ by
$ \langle\widehat{u}(\varphi), \psi\rangle =\langle
u(\varphi),\widehat{\psi}\rangle$
for all $\varphi\in \calD(\mR)$ and $\psi\in \calS(\mR^d)$.  
Here, for
$\psi \in \calS'(\mR^d)$, we define its Fourier transform
$\widehat{\psi}\in \calS'(\mR^d)$ by
$$
\widehat{\psi}(\bxi):=\int_{\mR^d} \psi(\bbx)e^{-i\langle \bxi, \bbx\rangle_{\mR^d}}\textrm{d}^d \bbx\quad (\bxi \in \mR^d),
$$
where $\langle \cdot,\cdot \rangle_{\mR^d}$ is the Euclidean inner
product on $\mR^d$.
\end{definition}  
 
\noindent 
   $\calL(\calD(\mR), \calS'(\mR^d))$ can be
considered to be a subspace of $\calD'(\mR^{d+1})$ as follows. 
 For an
element $u\in \calL(\calD(\mR), \calS'(\mR^d))$, define $U$ by
\begin{equation}
\label{6_December_2019_09:49}
\langle U, \varphi \otimes \psi\rangle = \langle u(\varphi),
\psi\rangle \quad (\varphi \in \calD(\mR), \;\psi \in \calD(\mR^d)\subset \calS(\mR^d)).
\end{equation}
 By the Schwartz kernel theorem \cite[Theorem~5.2.1, p.128]{Hor}, it follows that 
 there is a unique distribution $U\in \calD'(\mR^{d+1})$ such that \eqref{6_December_2019_09:49} is satisfied. The space 
$\calL(\calD(\mR), \calS'(\mR^d))$ is
also isomorphic to the completed projective- (or\footnote{The
  projective tensor product topology and the epsilon tensor product
  topology coincide here, since at least one of the two spaces, and in
  fact in our case, both $\calD'(\mR)$ and $\calS'(\mR^d)$, are
  nuclear.} epsilon-)tensor product
$\calD'(\mR)\widehat{\otimes}_\pi \calS'(\mR^d)$ of the spaces
$\calD'(\mR)$ and $\calS'(\mR^d)$.

We will study the set of null solutions with respect to the space of
distributions tempered in the spatial directions, and give an
algebraic-geometric characterization of those polynomials $p$ for which the
corresponding null solution space consists of just the zero
solution. Before we state our result, we give a few
algebraic-geometric definitions, and some motivation for arriving at
this condition.

\goodbreak

\begin{definition}[Variety]$\;$

\noindent 
Given a set $I$ of polynomials from $\mC[X_1,\cdots,X_d ]$, the {\em
  variety} $V(I)$ of $I$ in $\mC^{d}$, is the set of all common
zeros of the polynomials from $I$, that is,
$V(I)=\{\bxi\in \mC^d : p(\bxi)=0 \textrm{ for all } p\in I\}.  $
\end{definition}

\noindent We make the following two observations, leading  us
to our main result.
\begin{enumerate}
\item Let $p\in \mC[\bbX]$, and let $u\in \calS'(\mR^d)$ be such that
  $ D_p u=\mathbf{0}.  $ Taking Fourier transform, 
  $p(i\bxi) \widehat{u}=\mathbf{0}$.  So
  $
 \textrm{supp}(\widehat{u})\subset \{\bxi \in \mR^d: p(i\bxi)=0\}.
 $ 
 Thus  if $V(p)\cap i\mR^d=\emptyset$, then $\widehat{u}=\mathbf{0}$, and
 hence also $u=\mathbf{0}$.
\item Let $p\in \mC[T]$, and let $u\in \calD'(\mR)$ be such that
  $D_p u=\mathbf{0}$.  Then $u$ is a classical solution,
  expressible as a linear combination of the real analytic functions
  $t^ke^{\lambda t}$ for some nonnegative integers $k$ and complex
  numbers $\lambda$. If the past of $u$ is zero, that is,
  $u|_{t<0}=\mathbf{0}$, then $u=\mathbf{0}$.
\end{enumerate}

\noindent As our solution space
$\calL(\calD(\mR), \calS'(\mR^d))\simeq
\calD'(\mR)\widehat{\otimes}_\pi \calS'(\mR^d)$,
we expect our algebraic-geometric characterisation to reduce to above 
extreme cases when the polynomial belongs either to $\mC[\bbX]$ or to $\mC[T]$.  
In order to formulate this algebraic-geometric condition, we give
the following definition.

\begin{definition}[$\bbX$-content]$\;$

\noindent 
Let $p\in \mC[\bbX,T]$. Writing
 $
p=a_0+a_1 T+\cdots+ a_n T^n \in \mC[\bbX][T] ,
$ 
where $a_0,\cdots, a_n\in \mC[\bbX]$, the {\em $\bbX$-content 
$C_{\bbX}(p)$ of $p$} is the ideal in $\mC[\bbX]$ generated by
$a_0,\cdots, a_n$.
\end{definition}

\noindent 
We show below that if the variety $V(C_{\bbX}(p) )$ of the
$\bbX$-content of $p$ meets $i\mR^d$, then the null solution space in
$\calL(\calD(\mR), \calS'(\mR^d))$ associated with $p$ is nontrivial.

\begin{theorem}
\label{main_result_1}$\;$

\noindent 
Let $p\in \mC[\bbX,T]$.  If
$N_{\calL(\calD(\mR), \calS'(\mR^d))}(p)=\{\mathbf{0}\}$, then
$ V( C_{\bbX}(p) ) \cap i \mR^d= \emptyset$.
\end{theorem}
\begin{proof} Let 
  $V( C_{\bbX}(p) )\cap i \mR^d \neq \emptyset$, and $\bxi_0\in \mR^d$ be such that $ i \bxi_0 \in V(C_{\bbX}(p))$.
  Consider $u:=e^{i\langle \bbx,\bxi_0\rangle_{\mR^d}}\otimes \Theta$,
  where  $\Theta$ is any nonzero function in
  $\textrm{C}^\infty(\mR)$ which has a zero past, for example,
\begin{equation}
\label{eqn_3_Dec_2019_15:55}
\Theta(t)=\left\{\begin{array}{ll}
0 & \textrm{if } t\leq 0,\\
e^{-1/t} & \textrm{if } t> 0.
\end{array}
\right.
\end{equation}
Then $u\in \calL(\calD(\mR), \calS'(\mR^d))$ and it has zero past,
that is, $u|_{t<0}=\mathbf{0}$. If
 $
p=a_0 + a_1 T +\cdots+ a_n T^n,
$ 
where $a_0, \cdots, a_n\in \mC[\bbX]$, then
$a_0,\cdots, a_n\in C_{\bbX}(p)$, and so
 $
a_0(i\bxi_0)=\cdots= a_n(i\bxi_0)=0.
$ 

\noindent 
Consequently,
 $
D_p u= a_0(i\bxi_0) e^{i\langle \bbx, \bxi_0\rangle_{\mR^d} }
\otimes \Theta + \cdots +a_n(i\bxi_0) e^{i\langle \bbx, \bxi_0
  \rangle_{\mR^d}} \otimes \Theta^{(n)} =\mathbf{0}.
  $ 
Hence $u\in N_{ \calL(\calD(\mR), \calS'(\mR^d))}(p)$. But
$u\neq \mathbf{0}$, and so
$N_{ \calL(\calD(\mR), \calS'(\mR))}(p)\neq \{\mathbf{0}\}$.
\end{proof}

\noindent In light of the necessity of
$V( C_{\bbX}(p) ) \cap i \mR^d= \emptyset$ for $N_{\calL(\calD(\mR), \calS'(\mR^d))}(p)=\{\mathbf{0}\}$, a
natural question is whether this condition is also sufficient.  Our
main result (Theorem~\ref{main_result_3}) is to show the sufficiency. 
Thus, Theorems~\ref{main_result_1} and \ref{main_result_3} together give:

\begin{theorem}$\;$

\noindent 
Let $p\in \mC[\bbX,T]$.  
$N_{\calL(\calD(\mR), \calS'(\mR^d))}(p)=\{\mathbf{0}\}$ if and only if 
$ V( C_{\bbX}(p) ) \cap i \mR^d= \emptyset$.
 
\end{theorem}

\noindent In the last section, we will also consider distributions 
which have spatial profiles at each time instant lying in certain Besov spaces. 

\medskip 

\noindent 
We summarise the results in a table below:

\medskip 

\begin{center}
 \begin{tabular}{|c|c|c|c|}\hline
      & $\phantom{\displaystyle\sum}\!\!\!\!\!\!\!$ Solution space $S\!\!\!\!\!\!\!\phantom{\displaystyle\sum}$ & Test on $p$ for $N_S(p)=\{\mathbf{0}\}$ & Result reference \\ \hline \hline 
  (1) & $\phantom{\displaystyle\sum}\!\!\!\!\!\!\!$ $\textrm{C}^\infty (\mR^{d+1})\!\!\!\!\!\!\!\phantom{\displaystyle\sum}$ & $\deg(p)=\deg(p(\mathbf{0},T))$ & Proposition~\ref{prop_Hor} \\\hline 
  (2) & $\phantom{\displaystyle\sum}\!\!\!\!\!\!\!$ $\calD'(\mR^{d+1})\!\!\!\!\!\!\!\phantom{\displaystyle\sum}$ & $\deg(p)=\deg(p(\mathbf{0},T))$ & Proposition~\ref{prop_Hor} \\\hline 
  (3) & $\phantom{\displaystyle\sum}\!\!\!\!\!\!\!$ $\calD(\mR^{d+1})\!\!\!\!\!\!\!\phantom{\displaystyle\sum}$ & $p\neq \mathbf{0}$ & Proposition~\ref{prop_a}  \\ \hline 
  (4) & $\phantom{\displaystyle\sum}\!\!\!\!\!\!\!$ $\calE'(\mR^{d+1})\!\!\!\!\!\!\!\phantom{\displaystyle\sum}$ & $p\neq \mathbf{0}$ & Proposition~\ref{prop_a}  \\ \hline 
  (5) & $\phantom{\displaystyle\sum}\!\!\!\!\!\!\!$ $\calL(\calD(\mR),\calS'(\mR^d))\!\!\!\!\!\!\!\phantom{\displaystyle\sum}$ & $V(C_{\bbX}(p))\cap i\mR^d=\emptyset $ & 
  Theorems~\ref{main_result_1}, \ref{main_result_3} \\ \hline 
  (6) & $\phantom{\displaystyle\sum}\!\!\!\!\!\!\!$ $\calL(\calD(\mR),B_{p,q}(\mR^d))\!\!\!\!\!\!\!\phantom{\displaystyle\sum}$ & $p\neq \mathbf{0} $ & Theorem~\ref{main_result_sob} \\ \hline
  (7) & $\phantom{\displaystyle\sum}\!\!\!\!\!\!\!$ $\calL(\calD(\mR), H_s(\mR^d))\!\!\!\!\!\!\!\phantom{\displaystyle\sum}$ & $p\neq \mathbf{0}$ & Corollary~\ref{cor_sob} \\ \hline 
  (8) & $\phantom{\displaystyle\sum}\!\!\!\!\!\!\!$ $\calL(\calD(\mR), \calS(\mR^d))\!\!\!\!\!\!\!\phantom{\displaystyle\sum}$ & $p\neq \mathbf{0}$ & Corollary~\ref{cor_sob} \\ \hline 
  (9)& $\phantom{\displaystyle\sum}\!\!\!\!\!\!\!$ $\calL(\calD(\mR),\calE'(\mR^d))\!\!\!\!\!\!\!\phantom{\displaystyle\sum}$ & $p\neq \mathbf{0}$ & Theorem~\ref{theorem_last_one} \\ \hline 
 (10)& $\phantom{\displaystyle\sum}\!\!\!\!\!\!\!$ $\calD'_{\mA}(\mR^{d+1})\!\!\!\!\!\!\!\phantom{\displaystyle\sum}$ & $\forall\mathbf{v}\in A^{-1}2\pi\mZ^d$, 
 $  \exists t\in \mC$:  & Theorem~\ref{theorem_lalalast_one} \\ 
 & & $p( i \mathbf{v}, t) \neq 0.$ & \\ \hline 
 
 \end{tabular}
\end{center}

\medskip 

\noindent
The key idea used in proving the sufficiency part is as follows. 
By taking Fourier transform, the partial derivatives $\partial_{x_k}$ with
respect to the spatial variables $x_k$ are converted into $i\xi_k$, and so we obtain $p(i\bxi,\partial_t)\widehat{u}=\mathbf{0}$, 
a family, parameterised by $\bxi\in \mR^d$, of equations involving $\partial_t^k$ with the polynomial
coefficients $a_k( i \bxi)$. One would like to `freeze' a $\bxi\in \mR^d$, to get
an ODE for $(\widehat{u}(\cdot))(\bxi) \in \calD'(\mR)$, where for such a solution to an ODE we can indeed say that zero past implies  
zero future, and so the proof can be completed easily by varying the arbitrarily fixed $\bxi$. This is possible if the  
spatial Fourier transform is a function, so that the evaluation at $\bxi$ is allowed, and this is essentially how one shows the results (6)-(10). 

For showing our main result (5), where spatial Fourier transform will not result in a function of $\bxi$, but rather a distribution, 
the idea is as follows. Using Holmgren's uniqueness principle, the support of $\widehat{u}$ 
is contained in $V\times [0,\infty)$, where $V$ is the real zero set of the leading coefficient $a_n$. 
If $d=1$, so that $a_n$ were a polynomial of just one variable, then the real zeroes are isolated points, and we can complete the proof using 
a structure theorem of Schwartz, which says that distributions supported on a line must 
be essentially the Dirac delta and its derivatives, tensored with distributions $T_k$ of one variable (time). 
We can then boil down $p(i\bxi, \partial_t) \widehat{u}=\mathbf{0}$ to give an ODE for these distributions $T_k$ of time, and 
since each $T_k$ can be shown to have zero past, we can conclude that the $T_k$s must be zero. 
So this is how the proof works when $d=1$ and when $a_n$ was a polynomial of just one variable. In the general case, 
 to handle the case when $a_n$ may be a polynomial of $d$ variables, we proceed inductively on the number of spatial dimensions $d$. It is too much to hope that 
at each inductive step we end up with polynomials as coefficients of $\partial_t^k$, 
since polynomial parametrisations of the zero sets of the polynomial $a_n$ may not be possible (e.g. $\{(X,Y): X^2+Y^2-1=0\}$ does 
not possess a polynomial parametrisation). But the $d=1$ case just relied on the discreteness of the zero set of $a_n$, which is 
 also guaranteed if $a_n$ were real analytic instead of being a polynomial. 
 So to carry out the induction, we use the set up where we make sure that the 
 coefficients of $\partial_t^k$ obtained at each inductive step are 
 real analytic functions. To begin with, polynomials are real analytic, real analytic varieties do 
possess locally real analytic parametrisations (\L ojaciewicz structure theorem for real analytic varieties), 
and composition of real analytic functions is real analytic. 
This allows us to complete the induction step, by again appealing to a structure theorem of Schwartz, now 
for distributions with support in a smooth manifold. 
 The technical details are carried out in Lemma~\ref{main_lemma}. 
 
The organisation of the article is as follows.
\begin{itemize}
\item In Section~\ref{Section_preliminaries}, we recall some
  preliminaries needed for the proofs.
\item In Section~\ref{Section_technical_lemma}, we will prove the
  central technical result in Lemma~\ref{main_lemma}, which will lead
  to the proof of Theorem~\ref{main_result_3} on the sufficiency.
\item In Section~\ref{final_section}, we will prove
  Theorem~\ref{main_result_3}.
\item In Section~\ref{section_final}, we   consider distributions 
which have spatial profiles at each time instant lying in certain Besov spaces.
\item In Section~\ref{section_finalists_final} we consider distributions that are periodic in the spatial directions. 
\item Finally, in the last section, we mention a class of auxiliary open problems on the theme of null solutions. 
\end{itemize}
  
\section{Preliminaries}
\label{Section_preliminaries}

\noindent In this section, we recall three auxiliary known results
needed for the proof of Lemma~\ref{main_lemma}:

\smallskip 

$\bullet$  Holmgren's uniqueness theorem,

$\bullet$ Schwartz structure theorem for distributions supported on a
  manifold,

$\bullet$ \L ojaciewicz structure theorem for real analytic varieties.

\subsection{Holmgren's uniqueness theorem}
 
\noindent Before recalling this result, we first
we establish some terminology and notation.  Let $\Omega\subset \mR^d$ be
an open set, and consider the differential operator $P$ with real
analytic coefficients $a_{\mathbf{n}}$ of order $N$:
$$
P=P(x,{\bm{\partial}})=\sum_{|\mathbf{n}|\leq N} a_n(\mathbf{x})
\frac{\partial^{n_1}}{\partial x_1^{n_1}}\cdots
\frac{\partial^{n_d}}{\partial x_d^{n_d}}.
$$
Here, for a multi-index $\mathbf{n}=(n_1,\cdots, n_d)$ of nonnegative
integers, we define $|\mathbf{n}|:= n_1+\cdots+n_d$.  
We now recall the following version of the uniqueness theorem of Holmgren \cite[Lemma~5.3.2,
p.125]{HorII}:
  
  \goodbreak

\begin{proposition}[Holmgren's uniqueness theorem]$\;$
 
\noindent In an open subset  $\Omega$ of $\mR^d$, let  
$$
P(\bbx,{\bm{ \partial}})=\sum_{|\mathbf{n}|\leq N} a_n(\mathbf{x})
\frac{\partial^{n_1}}{\partial x_1^{n_1}}\cdots
\frac{\partial^{n_d}}{\partial x_d^{n_d}}
$$ 
be a differential operator having coefficients real analytic on $\Omega$.

\smallskip 

\noindent Assume that 
the coefficient of 
 $\displaystyle 
\frac{\partial^N}{\partial x_d^N} 
$
never vanishes in $\Omega$. 

\smallskip 

\noindent 
If $u\in \calD'(\Omega)$ and $P(\bbx,{\bm{ \partial}})u=0$ in  
$\Omega_c:=\{\bbx\in \Omega: x_d<c\}$ for some $c$, then $u=\mathbf{0}$ in $\Omega_c$ provided that 
$\Omega_c\cap (\textrm{\em supp}(u))$ is relatively compact in $\Omega$.
\end{proposition}
 
\noindent We will use the following consequence of this. 

\goodbreak 

\begin{lemma}
\label{lemma_Holmgren}$\;$

\noindent 
Let 

\medskip 

\noindent $\;\bullet$ 
$U$ be an open subset of $\mR^d$, 
  $\phantom{\displaystyle \frac{\partial^n}{t^n}}$

\noindent $\;\bullet$ $c_0,\cdots, c_N$ are real analytic functions of $d$ variables
  in $U$,$\phantom{\displaystyle \frac{\partial^n}{\partial t^n}}$

\noindent $\;\bullet$ $u\in \calD'(U\times \mR)$,
  $\phantom{\displaystyle \frac{\partial^n}{\partial t^n}}$

\noindent $\;\bullet$ 
  $u|_{t<0}=\mathbf{0}$,$\phantom{\displaystyle
    \frac{\partial^n}{\partial t^n}}$

\noindent $\;\bullet$ 
  $\displaystyle c_0(\bxi)u+c_1(\bxi)\frac{\partial u}{\partial
    t}+\cdots+ c_N(\bxi)\frac{\partial^N u}{\partial t^N}=\mathbf{0}$.

\medskip 
    
\noindent Then $\textrm{\em supp}(u)\subset \{(\bxi,t) \in U\times \mR:c_N(\bxi)=0\}$. 
\end{lemma}
\begin{proof} Let $\bxi_0\in U$ be such that $c_N(\bxi_0)\neq 0$. 
Let $r>0$ be such that the open ball $B(\bxi_0, 2r)\subset U$ and $c_N(\bxi)\neq 0$ in $B(\bxi_0,2r)$. 
 We will use Holmgren's uniqueness theorem with $\Omega_{\bxi_0}:=B(\bxi_0,2r) \times \mR$ and the differential operator
$$
P((\bxi,t),\bm{\partial})
:=
c_0(\bxi)
+c_1(\bxi)\frac{\partial}{\partial t}
+\cdots
+c_N(\bxi)\frac{\partial^N}{\partial t^N}.
$$
The coefficient $c_N(\bxi)$ of $ \displaystyle \frac{\partial^N}{\partial t^N}$ never vanishes in $\Omega_{\bxi_0}$. 

\smallskip 

\noindent Let $\widetilde{u}$ be the restriction of $u$ to $\Omega_{\bxi_0}=B(\bxi_0,2r)\times \mR\subset U\times \mR$. 
We already know that $\textrm{supp} (\widetilde{u})\subset B(\bxi_0,2r)\times [0,\infty)$ since $\widetilde{u}|_{t<0}=\mathbf{0}$ 
(which in turn follows from $u|_{t<0}=\mathbf{0}$). For any $c>0$, with $\Omega_c:=\{(\bxi,t):\bxi \in B(\bxi_0,r) \textrm{ and }t<c\}$, 
we have $\Omega_c \cap \textrm{supp}(\widetilde{u})$ is relatively compact in $\Omega_{\bxi_0}$. Hence $\widetilde{u}=\mathbf{0}$ in $\Omega_c$. 
As the choice of $c>0$ was arbitrary, it follows that $\widetilde{u}=\mathbf{0}$ in $B(\bxi_0,r)\times \mR$. By varying the $\bxi_0$ having the property that $c_N(\bxi_0)\neq 0$, 
we obtain that the restriction of $u$ to the set $V:=\{\bxi \in U: c_N(\bxi)\neq 0\}\times \mR$ is the zero distribution $\mathbf{0}\in \calD'(V)$. 
So  $\textrm{supp}(u)\subset \{(\bxi ,t)\in U\times\mR : c_N(\bxi)=0\}$. 
\end{proof}

\subsection{Schwartz structure theorem for distributions with support
  in a submanifold of $\mR^d$}
\label{Schwa}

We will need a local structure result, due to Schwartz
\cite[Theorem~XXXVII, page 102]{Sch}, for distributions with support
contained in a smooth submanifold of $\mR^d$ (analogous to the
well-known structure theorem saying that a distribution with support in
a point is a linear combination of the Dirac delta distribution at
that point and its derivatives). But before stating this result, it is
useful to keep the following guiding example in mind.

Consider in $\mR^d$ the manifold $M=S^{d-1}$, namely the unit
$(d-1)$-dimensional sphere. Let $\partial_r$ denote the radial partial
derivative. Then we expect that a distribution $T$ in $\mR^d$ having
its support in $S^{d-1}$ should be decomposable as
$$
T=\sum_{k=0}^K \partial_r^k T_k,
$$
for some distributions $T_k$ on $S^{d-1}$, where the action of the
right-hand side above on a test function
$\varphi=\varphi(r,\theta_1,\cdots, \theta_{d-1})$ (in an appropriate
chart) is understood as
 $$\displaystyle 
\sum_{k=0}^K (-1)^k \langle
T_k, \partial_r^k\varphi(r,\cdot)|_{r=1}\rangle.
$$
\noindent 
A generalisation of this is given below; see
\cite[Theorem~XXXVII, p. 102]{Sch}. Here, 
for a multi-index $\bbk=(k_{d'+1},\cdots, k_d)$ of nonnegative
integers, we define $|\bbk|=k_{d'+1}+\cdots+k_d$, and
$$
\partial_\bby^\bbk =\left(\frac{\partial}{\partial
    y_{d'+1}}\right)^{k_{d'+1}}\cdots \left(\frac{\partial}{\partial
    y_{d}}\right)^{k_{d}}.
$$
 
\begin{proposition}[Schwartz structure theorem]
$\;$
 
\noindent
Suppose that 

\medskip 

\noindent $\;\;\bullet\;$ $M$ is a submanifold of $\mR^d$ of dimension $d'$, 

\noindent $\;\;\bullet\;$ $\bxi_0\in M$, and 

\noindent $\;\; \bullet\;$ $\bby$ are  coordinates in
$B(\bxi_0,R)=\{\bxi\in \mR^d: \|\bxi-\bxi_0\|_2<R\}$ in $\mR^d$, 

\noindent \phantom{aal} such
that
$ B(\bxi_0,R)\cap M=\{\bxi\in B(\bxi_0,R): y_{d'+1}(\bxi)=\cdots =
y_{d}(\bxi)=0\}.  $

\medskip 

\noindent 
Then a distribution $T$ on $\mR^d$ with support in $M$ can be locally
decomposed as
$$
T=\sum_{|\bbk|\leq K} \partial_\bby^\bbk T_\bbk,
$$
for some distributions $T_\bbk$ on $M$. 
\end{proposition}
 
\subsection{\L ojasiewicz structure theorem for real analytic varieties}
\label{Loja}

\noindent We use the terminology and notation from \cite{KraPar}.  Let
$U\subset \mR^d$ be open. Then $\textrm{C}^\omega(U)$ denotes the
commutative ring (with respect to pointwise addition and
multiplication) of all (possibly complex-valued) real analytic
functions in $d$ real variables; see \cite[Definition~1.1.5,
p.3]{KraPar}.  In order to prove our main result in the form of the
technical result, namely Lemma~\ref{main_lemma}, we will need a
 structure theorem for real analytic varieties given in Lemma~\ref{baby_Loja} below.  
 Roughly speaking, this result says
that the zero set of a real analytic function of $d$-variables admits
a decomposition\footnote{This is analogous to a result due to Whitney
  \cite{Whi}, saying that a real algebraic variety $V$ can always be
  decomposed into a union of `algebraic partial manifolds' of
  decreasing dimensions: For example, if $V$ is the curve
  $V=\{(x,y)\in \mR^2: x^2-y^3=0\}$ (with a `cusp' at the origin),
  then we have $V=M_1\cup M_2$, where $M_1$ is the curve minus the
  origin, and $M_2=\{(0,0)\}$.} into the union of real analytic
manifolds of various dimensions $\leq d-1$. Lemma~\ref{baby_Loja} 
is a consequence of a more elaborate structure theorem for real analytic varieties 
 due to S. \L ojaciewicz
\cite{Loj} (see also \cite{BM}), which we first recall below. We quote this result
after recalling a few pertinent definitions.

A function $H(x_1,\cdots, x_{d-1}; x_d)$ of $d$ real variables is
called a {\em distinguished polynomial} if it has the form
\begin{eqnarray*}
H(x_1,\cdots, x_{d-1}; x_d)
&=&x_d^m+A_1(x_1,\cdots, x_{d-1})x_d^{m-1}+\cdots\\
&& + A_{m-1}(x_1,\cdots,x_{d-1})x_d+ A_m(x_1,\cdots, x_{d-1}),
\end{eqnarray*}
where each analytic $A_\ell$ vanishes at
$(x_1,\cdots, x_{d-1})=\mathbf{0}\in \mR^{d-1}$.

Since the polynomial
$$
\prod_{1\leq i<j\leq m} (X_i-X_j)^2\in \mZ[X_1,\cdots, X_m]
$$
is symmetric, the theorem on symmetric polynomials \cite[p.24]{Loj}
implies the existence of a unique polynomial
$\Delta_m\in \mZ[Y_1,\cdots, Y_m]$ such that
$$
\prod_{1\leq i<j\leq m} (X_i-X_j)^2=\Delta_m(\sigma_1,\cdots, \sigma_m),
$$
where 
$$
\sigma_i=(-1)^i \sum_{\nu_1<\cdots <\nu_i} X_{\nu_1}\cdots X_{\nu_i},
\quad i=1,\cdots, m.
$$
Let $R$ be a commutative unital ring. 
The {\em discriminant} of a monic polynomial 
 $
p=X^m+a_{m-1}X^{m-1}+\cdots+ a_m\in R[X],
$ 
 is defined to be $\Delta_m(a_1,\cdots, a_m)\in R$.

 \smallskip 
 
 \noindent 
We recall \cite[Theorem~6.3.3, p.168]{KraPar} below. 

\begin{proposition}[\L ojaciewicz structure theorem] $\;$

\noindent 
Let $f(x_1,\cdots, x_d)$ be a real analytic function in a
neighbourhood of the origin. We may assume that
$f(0,\cdots, 0, x_d)\nequiv 0$.  After a rotation of the coordinates
$x_1,\cdots, x_{d-1}$, if needed, there exist $\delta_j>0$,
$j=1,\cdots, d$, and distinguished polynomials
$ H_\ell^k(x_1,\cdots, x_k; x_\ell) \quad (0\leq k\leq d-1,\; k+1\leq
\ell\leq d) $
defined on $Q_k:=\{|x_j|<\delta_j, \;1\leq j\leq k\}$ such that the
discriminant $\Delta_\ell^k$ of $H_\ell^k$ does not vanish on $Q_k$
and the following properties are satisfied:
 
\smallskip 
 
\noindent {\em (1)} Each root $\zeta $ of
$H_\ell^k(x_1,\cdots,x_k;\cdot)$ on $Q_k$ satisfies
$|\zeta|<\delta_\ell$.
  
\smallskip 
  
\noindent {\em (2)} The set
$V:=\{\bbx=(x_1,\cdots, x_d):\forall j\; |x_j|<\delta_j\textrm{ and }
f(\bbx)=0\}$
has a decomposition $ V=M_{d-1}\cup \cdots \cup M_0.  $ The set $M_0$
is either empty or consists of the origin alone. For
$1\leq k\leq d-1$, we may write $M_k$ as a finite, disjoint union
$$
M_k=\bigcup_{\chi} \Gamma_\chi^k
$$
of $k$-dimensional sub-real analytic varieties having the following
description:
  
\smallskip 
  
\noindent $($\textrm{\em Real analytic parametrisation}$)$ Each $\Gamma_\chi^k$ is
defined by $d-k$ equations
\begin{eqnarray*}
    x_{k+1}= \chi_{\eta_{k+1}^k}(x_1,\cdots,x_k),
    &\cdots, &
    x_d= \chi_{\eta_{d}^k}(x_1,\cdots,x_k),
\end{eqnarray*}
where each $\chi_{\eta_{\ell}^k}$ is real analytic on an open subset
$\Omega_\chi^k\subset Q_k\subset \mR^k$,
$$
H_\ell^k(x_1,\cdots, x_k; \chi_{\eta^k_\ell})\equiv 0,
$$ 
and $\Delta_\ell^k\neq 0$ for all $(x_1,\cdots,x_k)\in \Omega_\chi^k$,
$\ell=k+1,\cdots, d$.
   
 \smallskip
%
%
    
\noindent $($\textrm{\em Stratification}$)$ For each $k$, the closure of $M_k$
contains all the subsequent $M_j$: that is,
$Q\cap M_k\supset M_{k-1}\cup \cdots \cup M_0$.
\end{proposition}

\noindent This result is stronger than what we need. We will only require the 
decomposition into lower dimensional real analytic varieties and the local analytic parametrisation. 
We state this as the following corollary of the above. 

\begin{lemma}
\label{baby_Loja}
Let $f:\mathbb{R}^d\rightarrow \mathbb{R}$ be a real analytic function, 
and 
$$
V:=\{\mathbf{x}\in \mathbb{R}^d: f(\mathbf{x}) =0\}
$$
be its variety. Then for each point $\mathbf{x}_0$ of $V$, there exists a neighbourhood  
 $\Omega$ such that there exists a decomposition $V\cap \Omega =M_{d-1}\cup \cdots \cup M_0$, 
where some of the $M_{d'}$s may be empty, and where the $M_{d'}$ are a 
real analytic varieties which are analytic submanifolds of $\mathbb{R}^d$ of dimension $d'$, 
admitting real analytic parametrisations as follows: 
For every $\mathbf{x}_0\in M_{d'}$, there exists a neighbourhood $U$ 
of $\mathbf{x}_0$ in $\mathbb{R}^d$ and a neighbourhood $W$ of 
$\mathbf{0}\in \mathbb{R}^d$, with a homeomorphism 
$\varphi:W\rightarrow U$ which is real analytic, 
$$
\mathbb{R}^{d'}\times \mathbb{R}^{d-d'}\supset W \owns (\btau,\bsigma)\mapsto 
\varphi(\btau,\bsigma)\in U,
$$
such that $M_{d'}\cap U=\{\varphi(\btau,\mathbf{0}): \btau\in \mathbb{R}^{d'} \text{ such that }(\btau,\mathbf{0})\in W\}$.
\end{lemma}
\begin{proof} This follows immediately from the above result, since (2) guarantees the 
local decomposition, and the real analytic parametrisation, namely 
$$
(x_1,\cdots, x_k)\mapsto (x_1, \cdots, x_k, \chi_{\eta_{k+1}^k}(x_1,\cdots,x_k),
    \cdots, 
    \chi_{\eta_{d}^k}(x_1,\cdots,x_k)) 
$$
corresponds to the one needed in the statement of the lemma if we take 
$\btau= (x_1,\cdots, x_k)$, $\bsigma=(x_{k+1},\cdots, x_d)$,  
and 
$$
\bxi(\btau,\bsigma)=(\btau,  \chi_{\eta_{k+1}^k}(x_1,\cdots,x_k)-x_{k+1},
    \cdots, 
    \chi_{\eta_{d}^k}(x_1,\cdots,x_k)-x_d).
$$
Then we note that the differential of  $\bxi$ has the form 
$$
d\bxi=\left[\begin{array}{cccc} I_{d'} & \mathbf{0} \\ 
             \mathbf{\ast} & -I_{d-d'} \end{array}\right],
$$
which is clearly invertible. Here $I_k$ denotes the $k\times k$ identity matrix with ones on the diagonal and zeroes elsewhere. 
\end{proof}

\goodbreak 

\section{The main technical lemma}
\label{Section_technical_lemma}

\noindent In this section, we will show the main technical result in
Lemma~\ref{main_lemma}, which will enable us to show our result on
sufficiency, namely Theorem~\ref{main_result_3}.

\begin{lemma}
\label{main_lemma}
Let 
\begin{enumerate} 
\item $U\subset \mR^d$ be open, 
  $\phantom{\displaystyle \frac{\partial^n}{\partial t^n}}$
\item
  $\mathbf{0}\neq p=c_0(\bxi)+c_1(\bxi)T+\cdots+c_n(\bxi) T^n \in
  \textrm{\em C}^\omega(U) [T]$, $c_n\neq \mathbf{0}$, 
  $\phantom{\displaystyle \frac{\partial^n}{\partial t^n}}$
\item $V(c_0,c_1,\cdots, c_n)\cap U =\emptyset$,
  $\phantom{\displaystyle\frac{\partial^n}{\partial t^n} }$
\item
  $w\in \calD'(U\times
  \mR)$,$\phantom{\displaystyle\frac{\partial^n}{\partial t^n} }$
\item
  $w|_{t<0}=\mathbf{0}$,$\phantom{\displaystyle
    \frac{\partial^n}{\partial t^n} }\displaystyle$
\item
\label{eqn_6_Dec_1_11:21}
  $\displaystyle c_0(\bxi)+c_1(\bxi) \frac{\partial}{\partial
    t}U+\cdots+ c_n(\bxi)\frac{\partial^n}{\partial t^n}
  w=\mathbf{0}$.
\end{enumerate}
Then $w=\mathbf{0}$.
\end{lemma}
\begin{proof} We prove this inductively on the number of spatial
  dimensions $d$.

\smallskip 

\noindent {\bf Step 1.} 
Let $d=1$. Holmgren's uniqueness theorem (Lemma~\ref{lemma_Holmgren}) implies 
$$
\textrm{supp}(w)\subset \{(\xi,t)\in U\times \mR : c_n (  \xi)=0,\;t\geq 0\}.
$$
If $c_n$ is constant (which must necessarily be $\neq 0$,
since $c_n$ was nonzero), then $w=\mathbf{0}$, and we are done.

Let $c_n$ be not a constant. Suppose that $w\neq \mathbf{0}$. Let $(\xi_k)_{k\in \mN}$ be the
real zeros of $c_n$ in $U$. Then each $\xi_k$ is isolated in $U$. We
have that
$$
\textrm{supp}(w)\subset \bigcup_{k\in \mN}\; \{\xi_k\}\times
[0,+\infty).
$$
Each of the half lines above carries a solution of the differential
equation 
$$
c_0(\bxi)+c_1(\bxi) \frac{\partial}{\partial
    t}U+\cdots+ c_n(\bxi)\frac{\partial^n}{\partial t^n}
  w=\mathbf{0},
$$
and $w$ is a sum of these.

Let $T\in (0,\infty)$. Take a $\xi_*\in \{\xi_1,\xi_2,\cdots\}$, $U$ a
neighbourhood of $\xi_*$ not containing the other $\xi_k$s, and an
$\alpha\in \calD(\mR)$ which is identically $1$ in a neighbourhood of
 $ [-T,T]$ such that the distribution $\alpha w\in \calD'(U\times \mR)$
is nonzero. Then $\alpha w$ has compact support, and by the structure
theorem for distributions (e.g. \cite[Theorem~2.3.5, p.47]{Hor} or the
result from Subsection~\ref{Schwa}), it follows that there exist
distributions $T_0,\cdots, T_K\in \calD'(\mR)$ (`of the time
variable'), with $T_K\neq \mathbf{0}$, such that
$$
\alpha w=\sum_{k=0}^K \Big(\Big( \frac{\partial}{\partial \xi}\Big)^k
\delta_{\xi_*} \Big) \otimes T_k.
$$
Here $\delta_{\xi_*}$ is the Dirac delta of the spatial variable
$\xi$, supported at $\xi_*$.  From the above, it can be shown that
also
\begin{equation}
\label{structure}
w=\sum_{k=0}^K \Big(\Big( \frac{\partial}{\partial \xi}\Big)^k \delta_{\xi_*} \Big) \otimes T_k
\end{equation}
in the strip $U\times  (-T,T)$. 

We claim that $T_K|_{(-T,0)}=\mathbf{0}$. For if not, then there is a
$\varphi\in \calD(\mR)$ with support in $ (-T,0)$ such that
$\langle T_K,\varphi\rangle \neq 0$. Hence the sum
$$
\sum_{k=0}^K \langle T_k ,\varphi \rangle \Big(
\frac{\partial}{\partial \xi}\Big)^k \delta_{\xi_*}
$$
is a nonzero distribution in $\calD'(U)$. Otherwise, we get the
contradiction that $\delta_{\xi_*} , \cdots,\delta_{\xi_*}^{(K)} $ are
linearly dependent in $\calD'(U)$. So there must exist a
$\psi\in \calD(U)$ such that
$$
\left\langle \sum_{k=0}^K \langle T_k ,\varphi \rangle \Big(
  \frac{\partial}{\partial \xi}\Big)^k \delta_{\xi_*},\;\psi
\right\rangle \neq 0,
$$
that is, $ \langle w, \psi\otimes \varphi \rangle \neq 0$. But the
support of $\psi\otimes \varphi$ is in $U\times (-T,0)$, and so we
have arrived at a contradiction to $w|_{t<0}=\mathbf{0}$. This proves
$T_K|_{(-T,0)}=\mathbf{0}$.

Using 
$$
\mathbf{0}=\sum_{\ell=0}^n c_\ell (\xi) \Big( \frac{\partial}{\partial
  t}\Big)^\ell w,
$$
we have that for $(\xi-\xi_*)^K\in \textrm{C}^\infty(U)$ and
$\varphi\in \calD(\mR)$, that
\begin{eqnarray}
\nonumber \!\!\!\!\!\!
0&\!\!\!=&\!\!\! 
\left\langle \sum_{\ell=0}^n c_\ell (\xi) 
\Big( \frac{\partial}{\partial t}\Big)^\ell\sum_{k=0}^K 
\Big( \frac{\partial}{\partial \xi}\Big)^k\delta_{\xi_*} \otimes T_k,
\;(\xi-\xi_*)^K \otimes \varphi \right\rangle
\\
\label{equation_15_nov_10:42}
&\!\!\!=&\!\!\! 
\sum_{\ell=0}^n \sum_{k=0}^K \left\langle \Big( 
\frac{\partial}{\partial t}\Big)^\ell T_k, 
\varphi \right\rangle (-1)^k \left\langle \delta_{\xi_*}, 
\Big( \frac{\partial}{\partial \xi}\Big)^k
\left(c_\ell(\xi)(\xi-\xi_*)^K\right)\right\rangle. 
\end{eqnarray}
But by the Leibniz product rule, 
$$
\Big( \frac{\partial}{\partial
  \xi}\Big)^k\left(c_\ell(\xi)(\xi-\xi_*)^K\right) = \sum_{m=0}^k
\mybinom{k}{m} \left((\xi-\xi_*)^K\right)^{(m)} \Big(
\frac{\partial}{\partial \xi}\Big)^{k-m} c_\ell(\xi) ,
$$
and if $k<K$, then for all $m=0,\cdots, k$, the $m$th derivative of
$(\xi-\xi_*)^K$ will be zero at $\xi=\xi_*$ as $K-m\geq K-k \geq
1$. So for $k<K$,
$$
\Big( \frac{\partial}{\partial
  \xi}\Big)^k\left(c_\ell(\xi)(\xi-\xi_*)^K\right)\Big|_{\xi=\xi_*} =
0.
$$
Hence the sum over $k=0,\cdots, K$ in \eqref{equation_15_nov_10:42}
collapses to one over $k=K$, giving
\begin{eqnarray*}
0&=& \sum_{\ell=0}^n  \left\langle   
\Big( \frac{\partial}{\partial t}\Big)^\ell T_K, 
\varphi \right\rangle (-1)^K 
       \left\langle \delta_{\xi_*}, 
\Big( \frac{\partial}{\partial \xi}\Big)^K
\left(c_\ell(\xi)(\xi-\xi_*)^K\right)\right\rangle 
\\
 &=&
\sum_{\ell=0}^n  \left\langle   
\Big( \frac{\partial}{\partial t}\Big)^\ell T_K, 
\varphi \right\rangle (-1)^K 
       c_\ell(\xi_*) K! 
\\
 &=& (-1)^K K! \left\langle 
\sum_{\ell=0}^n c_\ell(\xi_*) 
\Big( \frac{\partial}{\partial t}\Big)^\ell T_K, 
\varphi\right\rangle.
\end{eqnarray*}
As the choice of $\varphi\in \calD(\mR)$ was arbitrary, it follows
that
$$
\Big(c_0(\xi_*) +c_1(\xi_*) \frac{d}{dt} +\cdots+
c_n(\xi_*)\Big(\frac{d}{dt}\Big)^n \Big) T_K=\mathbf{0}.
$$
Owing to our condition that $V(c_0,c_1,\cdots, c_n)\cap U= \emptyset$,
we know that at least one of the coefficients
$c_0(\xi_*),\cdots, c_n(\xi_*)$ is nonzero\footnote{We know that
  $c_n(\xi_*)=0$ since $\xi_*$ was one of the roots of $c_n$.}. Thus
we now have a solution $T_K$ to an ODE with constant coefficients. But
then $T_K$ is a classical smooth solution expressible as a linear
combination of analytic functions of the type $t^ke^{\lambda t}$ for
some nonnegative integers $k$ and some complex numbers $\lambda$. The
zero past condition $T_K|_{(-T,0)}=\mathbf{0}$, furthermore implies
that this analytic function must in fact be identically zero, that is
$T_K=\mathbf{0}$ in $(-T,T)$, a contradiction.  Hence our assumption
that $w$ is nonzero can't be true.  Consequently, $w=\mathbf{0}$. This
completes the proof of the lemma when $d=1$.

\medskip 

\noindent {\bf Step 2.} Suppose now that $d>1$, and that the statement
of the lemma holds for all spatial dimensions strictly less than $d$.
We wish to prove the induction step that then the result holds for
$d$-many spatial dimensions too.  Let $w$ be a solution to
\begin{equation}
\label{eqn_289_nov_14:32}
c_0(\bxi)+c_1(\bxi) \frac{\partial}{\partial t}w+\cdots
+ c_n(\bxi)\frac{\partial^n}{\partial t^n} w=\mathbf{0},
\end{equation}
with zero past. Suppose that $w$ is nonzero. Holmgren's uniqueness
theorem (Lemma~\ref{lemma_Holmgren}) implies that
$$
\textrm{supp}(w)\subset \{(\bxi,t)\in U\times \mR : c_n ( \bxi)=0,\;t\geq 0\}.
$$
If $c_n$ is constant (which must necessarily be nonzero, since $c_n$
is nonzero), then $w=\mathbf{0}$, a contradiction, and so we are done.

Suppose that $c_n$ is not a constant. Then we can decompose $V(c_n) $ as 
$$
V(c_n):= M_{d-1}\cup \cdots \cup M_{0},
$$
where $M_k$ is the union of $k$-dimensional real analytic varieties, each possessing an
analytic parametrisation, as in Subsection~\ref{Loja}.

\begin{figure}[h]
     \center 
     \psfrag{X}[c][c]{$X_1$}
     \psfrag{Y}[c][c]{$X_d$}
     \psfrag{t}[c][c]{$\btau$}
     \psfrag{s}[c][c]{$\sigma$}
     \psfrag{U}[c][c]{$U$}
     \psfrag{W}[c][c]{$W$}
     \psfrag{x}[c][c]{$\bxi$}
     \psfrag{m}[c][c]{$\;\;\;M_{d-1}$}
     \psfrag{M}[c][c]{$M_{0}$}
     \psfrag{o}[c][c]{$\mathbf{0}$}
     \psfrag{V}[c][c]{$V(f)$}
     \psfrag{z}[c][c]{$\bxi_*$}
     \includegraphics[width=9 cm]{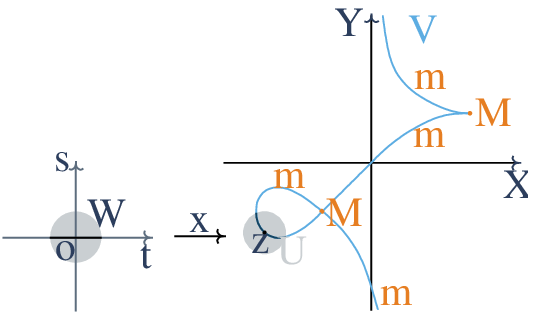}
\end{figure}

Let $\bxi_* \in M_{d'}$, where $0<d'\leq d-1$. Then there exists an
open neighbourhood $\Omega$ of $\bxi_*$, an open neighbourhood $W$ of
$\mathbf{0} \in \mR^{d}$, and a homeomorphism
$$
\mR^{d'}\times \mR^{d-d'}\supset W\owns (\btau,\bsigma)
\mapsto \bxi(\btau,\bsigma):W\rightarrow \Omega,
$$
such that $\btau\mapsto \bxi(\btau,\mathbf{0})
\in \textrm{C}^\omega(\widetilde{\Omega})$, where  
\begin{equation}
\label{eqn_Dec_3_2019_15:04} 
\widetilde{\Omega}
=\{ \btau\in \mR^{d'} \textrm{ such that }(\btau,\mathbf{0})\in W\},
\end{equation}
and
$ M_{d-1}\cap \Omega =\{\bxi(\btau,\mathbf{0}) : (\btau,\mathbf{0})\in
W\}$. Suppose that $w$ is nonzero in $\Omega\times \mR$. Then there is a
large enough $T>0$ such that $w$ is nonzero on $\Omega\times (-T,T)$.
By the Schwartz structure theorem from Subsection~\ref{Schwa}, we can
decompose $w$ locally in a neighbourhood $\Omega\times(-T,T)$ of
$(\bxi_*,0)\in \mR^{d+1}$ as
$$
w=\sum_{|\bbk |\leq K} \partial_{\bsigma}^{\bbk} T_{\bbk},
$$
for some distributions $T_{\bbk}$ on
$(M_{d'}\cap \Omega)\times(-T,T)$, such that not all
$T_{\bbk}=\mathbf{0}$ when $|\bbk|=K$, Moreover, as
$w|_{t<0}=\mathbf{0}$, it follows that $T_{\bbk}|_{t<0}=\mathbf{0}$
for each $\bbk$.  Now for a $d-d'$ tuple $\bbk=(k_{d'+1},\cdots, k_d)$
of nonnegative integers, with $|\bbk|=K$, let $\psi_{\bbk}$ be the
smooth function
$$
\psi_\bbk:=\sigma_1^{k_{d'}+1}\cdots \sigma_{d-d'}^{k_d}.
$$
Then 
$$
\partial_{\bsigma}^{\bbk'} \psi_{\bbk} \big|_{\bsigma=\mathbf{0}}
=k_{d'+1}!\cdots k_d! \cdot 
\delta_{k_{d'+1},  k'_{d'+1}}\cdots \delta_{k_d, k'_d},
$$
where $\delta_{\ell,\ell'}$ denotes the Kronecker delta, defined to be
$1$ if $\ell=\ell'$, and $0$ otherwise.  Using
$$ 
c_0(\bxi)+c_1(\bxi) \frac{\partial}{\partial t}w+\cdots+
c_n(\bxi)\frac{\partial^n}{\partial t^n} w=\mathbf{0},
$$
it follows that for all
$\varphi \in \calD(\widetilde{\Omega}\times (-T,T))$, where
$\widetilde{\Omega}$ is as defined in \eqref{eqn_Dec_3_2019_15:04}, we
have
\begin{eqnarray*}
0&=& 
\left\langle \sum_{\ell=0}^n \sum_{|\bbk'|\leq K} c_\ell (\bxi(\btau,\bsigma))
  \left(\frac{\partial}{\partial t}\right)^\ell 
\partial_{\bsigma}^{\bbk'} T_{\bbk'} ,\;\psi_{\bbk} \otimes \varphi \right\rangle
\\
&=& 
\sum_{\ell=0}^n  \left\langle   (-1)^K  k_{d'+1}!\cdots k_d! c_\ell(\bxi(\btau,\mathbf{0}))
\Big( \frac{\partial}{\partial t}\Big)^\ell T_{\bbk}, \varphi \right\rangle .
\end{eqnarray*}
But as
$\btau\mapsto \bxi(\btau,\mathbf{0})\in
\textrm{C}^\omega(\widetilde{\Omega})$,
and each $c_\ell\in \textrm{C}^\omega(U)$, it follows that their
(well-defined) composition
$\btau \mapsto c_\ell(\bxi(\btau,\mathbf{0}))$ is real analytic too.
Thanks to the assumption that
$V(c_0,c_1,\cdots, c_n )\cap U= \emptyset$, we also obtain in
particular that with
$\widetilde{c}_\ell(\btau):=c_\ell(\bxi(\btau, \mathbf{0}))$,
$\ell=0,1,\cdots, n$, and with
$$
p_0:= \sum_{\ell=0}^n \widetilde{c}_\ell(\bxi(\btau,0)) T^\ell \in \textrm{C}^\omega(\widetilde{\Omega})[T], 
$$
we have
$V(\widetilde{c}_0,\widetilde{c}_1,\cdots, \widetilde{c}_n)\cap
\mR^{d'}=\emptyset$,
and $D_{p_0} T_{\bbk}=\mathbf{0}$. We also recall from the above that
$T_{\bbk}|_{t<0}=\mathbf{0}$.  By the induction hypothesis, we conclude that $T_{\bbk}=\mathbf{0}$.  Repeating this
argument for each $\bbk$ satisfying $|\bbk|=K$, gives 
$T_{\bbk}=\mathbf{0}$ whenever $|\bbk|=K$. But this is a
contradiction. This means that
$M_{d'}\cap \textrm{supp}(w)=\emptyset$. As $d'$ such that
$0<d'\leq d-1$ was arbitrary, we conclude that
$ \textrm{supp}(w)\subset M_0.  $ But now we repeat the same argument
above from Step 1, when $w$ was supported on isolated lines, to
conclude that $\textrm{supp}(w)=\emptyset$, that is, $w=\mathbf{0}$.
This completes the induction step, and the proof of the lemma.
\end{proof}

\section{Proof of sufficiency}
\label{final_section}

\noindent In Theorem~\ref{main_result_1}, we had seen that the
condition $V( C_{\bbX}(p) ) \cap i \mR^d= \emptyset$ is necessary for
the triviality of the null solution space
$N_{\calL(\calD(\mR), \calS'(\mR^d))}(p)=\{\mathbf{0}\}$.  We now show
that this condition is also sufficient.

\begin{theorem}
\label{main_result_3} $\;$

\noindent 
Let $ p=a_0(\bbX)+a_1(\bbX)T+\cdots+ a_n(\bbX)T^n \in \mC[\bbX][T] $
such that $a_n\neq \mathbf{0}\in \mC[\bbX]$.

\noindent If $ V( C_{\bbX}(p) ) \cap i \mR^d= \emptyset$, then
$N_{\calL(\calD(\mR), \calS'(\mR^d))}(p)=\{\mathbf{0}\}$.
\end{theorem}
\begin{proof} 
  Suppose that $V(C_{\bbX}(p) )\cap i \mR^d = \emptyset$.  Let
  $u\in \calL(\calD(\mR), \calS'(\mR^d))$ be such that
  $u|_{t<0}=\mathbf{0}$, $D_p u=\mathbf{0}$ and such that
  $u\neq \mathbf{0}$.  Upon taking Fourier transformation on both
  sides of the equation $D_p u =\mathbf{0}$ with respect to the
  spatial variables, we obtain
\begin{equation}
\label{eq_ast_2}
a_0(i \bxi )\widehat{u} + a_1( i \bxi)\frac{\partial}{\partial
  t}\widehat{u} + \cdots+a_n( i \bxi) \Big(\frac{\partial}{\partial
  t}\Big)^n\widehat{u} =\mathbf{0}.
\end{equation} 
By Lemma~\ref{main_lemma}, this implies
$\widehat{u}=\mathbf{0}$. Taking the inverse Fourier transform yields
$u=\mathbf{0}$, completing the proof.
\end{proof}

\begin{example}[Diffusion equation revisited] Consider the diffusion
  equation
$$
\Big(\frac{\partial}{\partial t}-\Delta \Big)u=\mathbf{0},
$$
that is, $D_pu=\mathbf{0}$, where
$ p(\bbX,T)=T - (X_1^2+\cdots+X_d^2)$.

The constant polynomial $a_1=\mathbf{1}$ is nonzero, and so
$V(C_{\bbX}(p) )\cap i \mR^d=\emptyset$.  Theorem~\ref{main_result_3}
implies that $N_{\calL(\calD(\mR), \calS'(\mR^d))}(p)=\{\mathbf{0}\}$,
in conformity with our physical intuition.  \hfill$\Diamond$
\end{example}

\noindent Modern physics rejects the diffusion equation as an accurate
model of physical reality since it is not `Lorentz invariant',
admitting infinite propagation speeds. This can already be seen in the
case of classical solutions to the initial value problem to the
diffusion equation, where the solution is given by a (spatial)
convolution of the initial data $f$ with the Gaussian kernel, and so
for arbitrarily small time instants $t>0$ and at $\bbx=\mathbf{0}$, even arbitrarily far away
 initial data has an influence, which violates the special
relativistic tenet that nothing travels faster than the speed of
light.  With this in mind, we choose to illustrate our main theorem also with the Lorentz
invariant Klein-Gordon equation.

\begin{example}[Klein-Gordon equation on Minkowski space]
For $m\in \mR$, consider the equation 
$$
\Big(\frac{\partial^2}{\partial t^2}-\Delta +m^2\Big)u=\mathbf{0},
$$
that is, $D_pu=\mathbf{0}$, where
$ p(\bbX,T)=T^2 - (X_1^2+\cdots+X_d^2)+m^2$.

The constant polynomial $a_2=\mathbf{1}$ is nonzero, and so
$V(C_{\bbX}(p) )\cap i \mR^d=\emptyset$.  Theorem~\ref{main_result_3}
implies that $N_{\calL(\calD(\mR), \calS'(\mR^d))}(p)=\{\mathbf{0}\}$. 

We remark that Proposition~\ref{prop_Hor} also gives a sensible result 
in this case, since 
$$
\deg (p)=\deg(T^2 - (X_1^2+\cdots+X_d^2)+m^2)=2=\deg (T^2+m^2)=\deg(p(\mathbf{0},T)),
$$
and so $N_{\calD'(\mR^{d+1})}(p)= \{\mathbf{0}\}$ and
$N_{\textrm{C}^\infty(\mR^{d+1})}(p)= \{\mathbf{0}\}$. 

If $\eta_{\mu\nu}$ ($\mu,\nu=0,1,2,3$) 
are the Minkowski metric tensor components 
in the Cartesian/inertial coordinates, then the only Lorentz-invariant 
scalar linear constant coefficient differential operator one can build 
has the form 
$$
\displaystyle \sum_{n=0}^N c_n (\eta^{\mu\nu} \partial_\mu \partial_\nu)^n,
$$
where $[\eta^{\mu\nu}]$ denotes the inverse of the metric matrix $[\eta_{\mu\nu}]$, and 
$c_k\in \mC$. This corresponds to the polynomial 
$$
p=\displaystyle \sum_{n=0}^N c_n(T^2-(X_1^2+\cdots+X_d^2))^n,
$$
and so $\deg(p)=\deg(p(\mathbf{0},T))$ is {\em always} satisfied 
for such Lorentz invariant partial differential operators. Thus H\"ormander's Proposition~\ref{prop_Hor} is physically 
sound from the spacetime perspective of special relativity.
\hfill$\Diamond$
 \end{example}

 \begin{example}Consider the equation 
 $$
 \frac{\partial^{d+1}}{\partial t \partial x_1\cdots \partial x_d} u=\mathbf{0},
$$
that is, $D_p u=\mathbf{0}$, where $p(\bbX,T)= X_1\cdots X_d T$.
 Then 
 $$
 V(C_{\bbX}(p))=\bigcup\limits_{k=1}^d \textrm{span}_{\mC}(\bbe_k),
 $$
where $\bbe_k$ is the standard basis vector in $\mC^d$ with $1$ in the $k$th entry, and all 
other entries zeroes. Thus $V(C_\bbX (p))\cap i\mR^d\neq \emptyset$, so that 
$$
N_{\calL(\calD(\mR), \calS'(\mR^d))}(p)\neq \{\mathbf{0}\}.
$$
This is  expected, and we can easily 
construct nontrivial null solutions as in the proof of Theorem~\ref{main_result_1}. In fact, with $\mathbf{1}$ denoting the 
 constant function on $\mR^d$ taking value $1$ everywhere, we can take  $u:=\mathbf{1}\otimes \Theta(t)$, where $\Theta$ 
 is as in \eqref{eqn_3_Dec_2019_15:55}. Then $\mathbf{0}\neq u\in N_{\calL(\calD(\mR), \calS'(\mR^d))}(p)$ since it 
 has zero past and satisfies $D_pu=\mathbf{0}$. 
\hfill$\Diamond$
\end{example}

\section{Spatial profile in Besov spaces} 
\label{section_final}

\noindent We mention that besides the space $\calL(\calD(\mR), \calS'(\mR^d))$, 
 one may consider also other natural solution spaces with some growth restriction in the spatial directions. 
 As an example, we consider $\calL(\calD(\mR), B_{p,k}(\mR^d))$, where $B_{p,k}(\mR^d)$ is 
 a subspace of $\calS'(\mR^d)$ defined below. 

We follow \cite[\S 10.1]{HorV2}. A positive function $k$ defined on
$\mR^d$ will be called a {\em temperate weight function} if there
exist positive constants $C$ and $N$ such that 
$$
k(\bxi+\beeta)\leq (1+C\| \bxi\|_2 )^N k(\beeta), \quad \bxi,\;\beeta\in \mR^d,
$$
where $\|\cdot\|_2$ denotes the Euclidean norm on $\mR^d$. 
The set of all such functions will be denoted by $\calK$. If $k\in \calK$ and
$1\leq p\leq \infty$, then the {\em Besov space}  $B_{p,k}$ is the set of all
distributions $u\in \calS'(\mR^d)$ such that the Fourier transform
$\widehat{u}$ of $u$ is a function and 
$$
\|u\|_{p,k}=\Big(\int_{\mR^d} |k(\bxi)\widehat{u}(\bxi)|^p \textrm{d}^d\bxi \Big)^{1/p}<\infty.
$$
When $p=\infty$, we take $\|u\|_{p,k}$ as $\textrm{ess.sup
}|k(\cdot)\widehat{u}(\cdot)|$. Then $B_{p,k}$ is a Banach space with
the above norm. The usual scale of Sobolev space $H_{s}(\mR^d)$
parameterised by real numbers $s$ corresponds to the class
$$
\calK_{\scriptscriptstyle\textrm{Sob}}:=\{k_s:s\in\mR\}\subset \calK,
$$
where 
$$
k_s(\bxi):=(1+\|\bxi\|^2)^{s/2}.
$$
We can think
of the space $\calL(\calD(\mR), B_{p,k}(\mR^d))$ as a subspace of
$\calD'(\mR^{d+1})$: if $u\in \calL(\calD(\mR), B_{p,k}(\mR^d))$, then
we define the distribution $U\in \calD'(\mR^{d+1})$ by 
$$
\langle U,\psi \otimes \varphi\rangle=\int_{\mR^d} \left(\langle u,\varphi\rangle\right)(\bxi) \psi(  \bxi) \textrm{d}^d\bxi
$$
for $\varphi \in \calD(\mR)$, $\psi \in \calD(\mR^d)$. 

We prove the following result. Despite again using the Fourier transform as the main tool, 
akin to the proof of Lemma~\ref{main_lemma}, the proof is markedly simpler, thanks to the 
possibility of `evaluation' at $\bxi$ (since for every `time' test function $\varphi \in \calD(\mR)$, we have that $\widehat{u}(\varphi)$ 
is a function of the variable $\bxi\in \mR^d$). 

\begin{theorem}
\label{main_result_sob}$\;$

\noindent 
Let $p\in \mC[\bbX,T]$.  Then $N_{\calL(\calD(\mR),
  B_{p,k}(\mR^d))}(p)=\{\mathbf{0}\}$ if and only if $p\neq \mathbf{0}$.
\end{theorem}
\begin{proof}
 `Only if' part:
  Suppose that $p=\mathbf{0}$. Then we just take any nonzero $\psi \in
  B_{p,k}$. For example, any nonzero $\psi\in\calD(\mR^d)$ will do. 
  Moreover, let $\Theta$ be the nonzero function in $\textrm{C}^\infty(\mR)$
  which has a zero past given in \eqref{eqn_3_Dec_2019_15:55}. 
Define $u$ by $u(x,t):=\psi(x) \Theta(t)$ for $x\in \mR^d$ and $t\in
\mR$. Then  we have that $u\in \calL(\calD(\mR), B_{p,k}(\mR^d))$, $u|_{t<0}=\mathbf{0}$ and
$D_pu=\mathbf{0}$. But $u\neq \mathbf{0}$, and so $N_{\calL(\calD(\mR), 
  B_{p,k}(\mR^d))}(p)\neq \{\mathbf{0}\}$.

\medskip

\noindent `If' part: Suppose that $p\neq \mathbf{0}$.  Let $u\in
\calL(\calD(\mR), B_{p,k}(\mR^d))$ be such that $u|_{t<0}=\mathbf{0}$ and $D_p
u=\mathbf{0}$. Let
 $
p= a_0+a_1 T+\cdots+ a_n T^n \in \mC[\bbX][T] ,
$ 
where $a_0,a_1,\cdots, a_n\in \mC[\bbX]$ and $a_n\neq
\mathbf{0}$ in $\mC[\bbX]$. Upon taking Fourier transformation on
both sides of the equation $D_p u =\mathbf{0}$ with respect to the spatial
variables, we obtain
\begin{equation}
\label{eq_ast'}
a_0(i \bxi)\widehat{u} + a_1(i\bxi)\frac{\partial}{\partial
  t}\widehat{u} + \cdots+a_n(i\bxi) \Big(\frac{\partial}{\partial
  t}\Big)^n\widehat{u} =\mathbf{0}.
\end{equation}
But if we fix a $\bxi\in \mR^d$ such that $a_n(i \bxi)\neq 0$, then it
follows that $(\widehat{u}(\varphi))(\bxi)=0$ for all $\varphi \in \calD(\mR)$. 
Since the Lebesgue measure of the set of zeros of the polynomial function $a_n(i\bxi)$ is zero, it
follows that for each $\varphi \in \calD(\mR)$, the function $\mR^d \owns \bxi \mapsto (\widehat{u}(\varphi))(\bxi)$ almost everywhere, and 
so $\widehat{u}(\varphi)=\mathbf{0}$. But then $\widehat{u}=\mathbf{0}$ too, and so $u=\mathbf{0}$. This
completes the proof.
\end{proof}

\noindent 
And immediate consequence is the following, since the $H_s(\mR^d)$ are all special instances of 
the Besov spaces $B_{p,q}(\mR^d)$ \cite[Example~10.1.2, p.5]{HorV2}, and 
also $\calS(\mR^d)\subset B_{p,k}(\mR^d)$ \cite[Thm.~10.1.7, p.7]{HorV2}.

\begin{corollary}
\label{cor_sob}$\;$

\noindent 
Let $p\in \mC[\bbX,T]$ and $S=\calL(\calD(\mR), H_s(\mR^d))$ or $S=\calL(\calD(\mR), \calS(\mR^d))$. 

\noindent 
Then $N_{S}(p)=\{\mathbf{0}\}$ if and only if $p\neq \mathbf{0}$.
\end{corollary}

\noindent 
By the Payley-Wiener-Schwartz theorem ~\cite[Prop. 29.1, p. 307]{Tre}, we know that the Fourier transform of elements of $\calE'(\mR^d)$ 
 can be extended to entire functions on $\mC^d$. Thus the same proof, mutatis mutandis, 
 as that of Theorem~\ref{main_result_sob} gives the following. 
 
 \begin{theorem}
  \label{theorem_last_one}
  $\;$

\noindent 
Let $p\in \mC[\bbX,T]$.  Then $N_{\calL(\calD(\mR),
  \calE'(\mR^d))}(p)=\{\mathbf{0}\}$ if and only if $p\neq \mathbf{0}$.
\end{theorem}

\section{Spatially periodic distributions}
\label{section_finalists_final}

\noindent In this final section, we consider the space $\calD'_{\mA}(\mR^{d+1})$, which is, roughly speaking, 
the set of all distributions on $\mR^{d+1}$ that are  
periodic in the spatial directions with a discrete set $\mA$ of periods. The precise definition of 
$ \calD'_{\mA}(\mR^{d+1})$ is given below. 

 For ${\mathbf{a}}\in \mR^d$, the {\em translation operation} 
${\mathbf{S_a}}$ on distributions in $\calD'(\mR^d)$ is
defined by
$$
\langle {\mathbf{S_a}}(T),\varphi\rangle=\langle T,\varphi(\cdot+{\mathbf{a}})\rangle\;
\textrm{ for all }\varphi \in \calD(\mR^d).
$$
A distribution $T\in \calD'(\mR^d)$ is said to be {\em periodic with a period} 
$\mathbf{a}\in \mR^d$  if $T= {\mathbf{S_a}}(T)$.  

Let $\mA:=\{{\mathbf{a_1}}, \cdots, {\mathbf{a_d}}\}$ be a linearly independent set vectors in $\mR^d$. 
We define $\calD'_{\mA}(\mR^d)$ to be the set of all distributions $T$ that satisfy 
$$
{\mathbf{S_{a_k}}}(T)=T, \quad k=1,\cdots, d.
$$
From \cite[\S34]{Don}, $T$ is a tempered distribution, and 
from the above it follows by taking Fourier transforms that 
 $
(1-e^{i {\mathbf{a_k}} \cdot \bxi})\widehat{T}=0$ for $ k=1,\cdots,d.
$ It can be seen that 
$$
\widehat{T}=\sum_{\mathbf{v} \in A^{-1} 2\pi \mZ^d} \alpha_{\mathbf{v}}(T) \delta_{\mathbf{v}},
$$
for some scalars $\alpha_{\mathbf{v}}\in \mC$, and where  $A$ is the matrix with its rows equal to 
the transposes of the column vectors ${\mathbf{a_1}}, \cdots, {\mathbf{a_d}}$:
$$
A:= \left[ \begin{array}{ccc} 
    \mathbf{a_1}^{\top} \\ \vdots \\ \mathbf{a_d}^{\top} 
    \end{array}\right].
$$
By the Schwartz Kernel Theorem \cite[p.~128, Theorem~5.2.1]{Hor},  we know that 
$\calD'(\mR^{d+1})$ is isomorphic as a topological space to $\calL(\calD(\mR), \calD'(\mR^d))$, 
the space of all continuous linear maps from $\calD(\mR)$ to  $\calD'(\mR^d)$, thought of as vector-valued distributions. 
In this section, we indicate this isomorphism by putting an arrow on top of elements of 
$\calD'(\mR^{d+1})$. Thus for $u\in \calD'(\mR^{d+1})$, we set $\vec{u}\in \calL(\calD(\mR), \calD'(\mR^d))$ 
to be the vector valued distribution defined by 
$$
\langle \vec{u}(\varphi), \psi\rangle=\langle u, \psi \otimes \varphi\rangle 
$$
for $\varphi\in \calD(\mR)$ and $\psi \in \calD(\mR^d)$. We define 
$$ 
\calD'_{\mA}(\mR^{d+1})=\{u\in \calD'(\mR^{d+1}): \textrm{ for all }\varphi \in \calD(\mR), \;
\vec{u}(\varphi) \in \calD'_{\mA}(\mR^d)\}.
$$
Then for $u\in \calD'_{\mA}(\mR^{d+1})$, 
$$
 \frac{\partial}{\partial x_k}u \in  \calD'_{\mA}(\mR^{d+1}) \;\textrm{ for } k=1, \cdots, d, \;
\textrm{ and } \;\; 
\frac{\partial}{\partial t}u \in  \calD'_{\mA}(\mR^{d+1}). 
$$
Also, for $u\in \calD'_{\mA}(\mR^{d+1})$, we define $\widehat{u}\in \calD'(\mR^{d+1})$ 
by 
$$
\langle \widehat{u}, \psi \otimes \varphi \rangle
=
\langle \vec{u}(\varphi), \widehat{\psi} \rangle,
$$
for $\varphi \in \calD(\mR)$ and $\psi \in \calD(\mR^d)$.

We have the following characterisation for the space of null solutions to be trivial. 

\begin{theorem}
 \label{theorem_lalalast_one}$\;$
 
 \noindent 
Suppose that $\mA\!=\!\{{\mathbf{a_1}}, \cdots, {\mathbf{a_d}}\}$ is a linearly independent set of vectors in $\mR^d$. 

\noindent Let $S=\calD'_{\mA}(\mR^{d+1})$ and $p \in \mC[\bbX, T]$.

\noindent 
Then $N_{S}(p)=\{\mathbf{0}\}$ if and only if for all $\mathbf{v}\in A^{-1}2\pi \mZ^d$, 
 there exists a $t \in \mC$ such that $p(  i \mathbf{v}, t)\neq 0 $. 
\end{theorem}
\begin{proof} $\;$ 

\noindent {\bf `Only if' part:} Let $\mathbf{v_0}\in A^{-1}2\pi \mZ^d$ be such that for all $t\in \mC$, $p(i \mathbf{v_0}, t) =\mathbf{0}$. 
Then $p ( i\mathbf{v_0},T)$ is the zero polynomial in $\mC[T]$. 
 Let  $\Theta \in C^\infty(\mR)$ be any nonzero smooth function such that $\Theta|_{t<0}=\mathbf{0}$. 
Define $u:=e^{ i  \mathbf{v_0} \cdot \mathbf{x}}\otimes \Theta$. Here  $\mathbf{v_0}\cdot\mathbf{x}$ is the usual Euclidean 
inner product of $\mathbf{v_0}$ and $\mathbf{x}$ in the real vector space $\mR^d$.  
Then we have that $u\in \calD'_{\mA}(\mR^{d+1})$, since  
$$
\mathbf{S_{a_k}} u= e^{ i  \mathbf{v_0}\cdot (\mathbf{x}+\mathbf{a_k})}\otimes \Theta
=e^{i  \mathbf{v_0}\cdot \mathbf{a_k} }e^{i \mathbf{v_0}\cdot \mathbf{x}}\otimes \Theta
= 1\cdot e^{ i \mathbf{v_0}\cdot  \mathbf{x} }\otimes \Theta= u.
$$
Moreover, $u$ has zero past, that is, $u|_{t<0}=\mathbf{0}$ because $\Theta|_{t<0}=\mathbf{0}$. 
Also, the trajectory $u\in N_{\calD'_{\mA}(\mR^{d+1})}(p)\setminus \{\mathbf{0}\}$ because $\Theta \neq \mathbf{0}$ and 
 $$
p\left(\frac{\partial}{\partial x_1},\cdots, \frac{\partial}{\partial x_d} , \frac{\partial}{\partial t}\right) 
 u= e^{i \mathbf{v_0}\cdot \mathbf{x}}  p\left( i \mathbf{v_0} , \frac{d}{dt} \right)  \Theta
=e^{ i \mathbf{v_0}\cdot \mathbf{x}} \cdot \mathbf{0}=\mathbf{0}.
$$ 
 This completes the proof of the `only if' part. 

\bigskip 

\noindent  (`If' part:) On the other hand, now suppose that for each 
$\mathbf{v}\in A^{-1}2\pi \mZ^d$, there exists a $t\in \mC$ such that 
$p(  i \mathbf{v}, t)\neq 0 $. Then the polynomial $p( i\mathbf{v},T)$ is not the zero polynomial in $\mC[T]$.
 So  $N_{\calD'(\mR)}(p( i \mathbf{v},T))=\{\mathbf{0}\}$. Thus for each 
$\mathbf{v}\in A^{-1}2\pi \mZ^d$, whenever a distribution 
$T\in \calD'(\mR)$ is such that $T|_{t<0}=\mathbf{0}$ and satisfies 
$$
p\left( i \mathbf{v},\frac{d}{dt}\right)T=\mathbf{0},
$$
there holds that $T=\mathbf{0}$. Now suppose that $u\in \calD'_{\mA}(\mR^{d+1})$ satisfies 
$u|_{t<0}=\mathbf{0}$ and 
\begin{equation}
 \label{eq_pf_if_part}
p\left(\frac{\partial}{\partial x_1},\cdots, \frac{\partial}{\partial x_d} , \frac{\partial}{\partial t}\right) 
 u=\mathbf{0}.
\end{equation} 
 Upon taking Fourier transformation on both
sides of the equation \eqref{eq_pf_if_part} with respect to the spatial
variables, we obtain
\begin{equation}
\label{eq_ast}
p\left(i \mathbf{\bxi}, \frac{\partial}{\partial t} \right) \widehat{u} =\mathbf{0}.
\end{equation}
For each fixed $\varphi \in \calD(\mR)$, $\vec{u}(\varphi)\in \calD_{\mA}'(\mR^d)$, and so it 
follows that 
\begin{equation}
\label{eq_ast2}
\reallywidehat{\overrightarrow{u}(\varphi)}=\sum_{\mathbf{v} \in A^{-1} 2\pi \mZ^d} \delta_{\mathbf{v}} \;\! \alpha_{\mathbf{v}}(\widehat{u},\varphi) ,
\end{equation}
for appropriate coefficients $\alpha_{\mathbf{v}}(\widehat{u},\varphi)\in \mC$. In particular, 
we see that the support of $\widehat{u}$ is contained in $A^{-1} 2\pi \mZ^d \times [0,+\infty)$. 
Thus each of the half lines in $A^{-1} 2\pi \mZ^d \times [0,+\infty)$ carries a solution of the differential
equation \eqref{eq_ast}, and $\widehat{u}$ is a sum of these. We will
show that each of these summands is zero. It can be seen from \eqref{eq_ast2} that the map 
$\varphi \mapsto  \alpha_{\mathbf{v}}(\widehat{u},\varphi) :\calD(\mR) \rightarrow \mC$ defines a distribution 
$T^{(\mathbf{v})}$ in $\calD'(\mR)$. Moreover, the support of $T^{(\mathbf{v})}$ is contained in $[0,+\infty)$.  
From \eqref{eq_ast2}, we see that for a small enough neighbourhood $N$ of $\mathbf{v} \in A^{-1} 2\pi \mZ^d$ in $\mR^d$, we have 
$$
\delta_{\mathbf{v}} \otimes p\left( i \mathbf{v}, \frac{d}{d t} \right)  T^{(\mathbf{v})}=\mathbf{0}
$$
in $N\times \mR$. 
But as we had seen above, our algebraic hypothesis implies that the set $N_{\calD'(\mR)} (p( i \mathbf{v},T))$ of  null solutions is trivial,   i.e. 
$N_{\calD'(\mR)} (p( i \mathbf{v},T))=\{\mathbf{0}\}$, and so $T^{(\mathbf{v})}=\mathbf{0}$. 
As this happens with each $\mathbf{v}\in A^{-1}2\pi \mZ^d$, we conclude that $\widehat{u}=\mathbf{0}$ and hence 
also $u=\mathbf{0}$. Consequently,  $N_{\calD'_{\mA}(\mR^{d+1})}(p)=\{\mathbf{0}\}$.
\end{proof}

\section{Open question: For which $p$ is the set of futures of null solutions dense in the set of futures of all solutions?} 

\noindent It follows from \cite{Hor3} that the set of futures of smooth null solutions, namely 
$$
\{u|_{t>0}: u\in N_{\textrm{C}^\infty(\mR^{d+1})}(p)\} 
$$
is dense in the set of futures of all smooth solutions, namely 
$$
\{u|_{t>0}: u\in \textrm{C}^\infty(\mR^{d+1}) \textrm{ and } D_p u=\mathbf{0}\}
$$
if each irreducible factor $p'$ of $p$ satisfies  $\deg(p')\neq \deg(p'(\mathbf{0},T))$. 
 
 In our alternative solution spaces, one could ask a 
similar question, namely if it is possible to give a characterisation in terms of the polynomial $p$ 
so that the set of futures of null solutions is dense in the set of futures of all solutions. 
We leave this class of an open questions for future investigation.

\medskip

\noindent {\bf Acknowledgements:} $\;$

\noindent I am grateful to Sara Maad Sasane and 
Arne Meurman (Lund University) for useful discussions.

\goodbreak


\begin{thebibliography}{99}

\bibitem{BM}
E. Bierstone and P. Milman. 
{Semianalytic and subanalytic sets}. 
{\em Institut des Hautes \'{E}tudes Scientifiques, Publications Math\'{e}matiques}, 
67:5-42, 1988.
      
\bibitem{Don}
W.F. Donoghue, Jr. 
{\em Distributions and Fourier Transforms}. 
Pure and Applied Mathematics 32, Academic Press, New York and London, 1969.
 
\bibitem{Hel}
G. Hellwig. 
{\em Partial differential equations. An introduction. Second edition.} 
B.G. Teubner, 1977.

\bibitem{Hor0}
L. H\"{o}rmander. 
On the theory of general partial differential operators. 
{\em Acta Mathematica}, 94:161-248, 1955.

\bibitem{Hor3}
L. H\"{o}rmander. 
Null solutions of partial differential equations. 
{\em Archive for Rational Mechanics and Analysis},
255-261 no. 4, 1960.

\bibitem{HorII} 
L. H\"ormander. 
{\em Linear Partial Differential Operators}. 
Fourth printing. Springer, 1976. 

\bibitem{Hor} 
L. H\"ormander.  
{\em The Analysis of Linear Partial Differential Operators. I.} Second edition. 
Springer, 1990.

\bibitem{HorV2}
L. H\"{o}rmander. 
{\em The analysis of linear partial differential operators. {II}}. 
Reprint of the 1983 original. Springer, 2005.

\bibitem{KraPar}
S. Krantz and H. Parks. 
{\em A Primer of Real Analytic Functions}. Second edition. 
Birkh\"auser, 2002. 

\bibitem{Loj}
S. \L ojasiewicz. 
{\em Introduction to Complex Analytic Geometry}. 
Translated from the Polish by Maciej Klimek. 
Birkh\"{a}user, 1991. 



\bibitem{Sch}
L. Schwartz. 
{\em Th\'{e}orie des Distributions}. 
Hermann, Paris, 1966. 

\bibitem{Tre}
F. Tr\`eves. 
{\em Topological Vector Spaces, Distributions and Kernels}. 
Dover, New York, 2006.

\bibitem{Whi}
H. Whitney. 
Elementary structure of real algebraic varieties. 
{\em Annals of Mathematics, Second Series},  66:545-556, 1957.

\end{thebibliography}
\end{document}